\documentclass{amsart}
\usepackage{amsmath,amssymb,amscd}

\newtheorem{theorem}{Theorem}
\newtheorem{lemma}[theorem]{Lemma}
\newtheorem{definition}[theorem]{Definition}

\newtheorem{remark}[theorem]{Remark}

\newcommand{\C}{\mathbb{C}}

\newcommand{\T}{\mathbb{T}}
\newcommand{\Pic}{\operatorname{Pic}}

\begin{document}

\title{A New CR Invariant for Contact 3-Manifolds and Classes of Open Books}
\author{Ali M. Elgindi}
\address{Henan Academy of Sciences}
\email{ali.m.elgindi@gmail.com}

\maketitle

\begin{abstract}
This paper introduces a new CR invariant for co-oriented contact structures on closed, orientable 3-manifolds. The invariant, which we denote as $\mu_M(\xi)$, takes values in the Picard group of complex line bundles $\Pic_{\C}(M)$. The construction associates to a contact structure $\xi$ and a supporting open book decomposition an embedding into $\C^3$, where the contact structure becomes the holomorphic line field along the binding. Using Stein theory, the induced holomorphic line bundle extends to all of $\C^3$ but we consider only its restriction to $M$. By Giroux's correspondence, we prove this construction is independent of the choice of open book, yielding a well-defined invariant $\mu_M(\xi) \in \Pic_{\C}(M)$ over the manifold. As an application, we distinguish two tight contact structures on the 3-torus $\T^3$ by showing their first Chern classes are different.
\end{abstract}

\section{Introduction}

We use our results in [1] to determine new CR-invariants of real contact structures on a real contact 3-manifold, which we will consider as being embedded naturally in complex 3-space.

\begin{theorem}
There exists a CR-invariant of real contact structures taking values in the Picard group of isomorphism classes of complex line bundles, determined by an open book decomposition of $M$ corresponding to the contact structure $\xi$, independent of the choice of supporting book. We denote this invariant by
\[
\mu_M(\xi) \in \Pic_{\C}(M).
\]
\end{theorem}

\begin{remark}
A CR invariant is generally stronger than an isotopy invariant. Two contact structures that are isotopic (hence have the same isotopy invariants) may have different CR structures and thus be distinguished by a CR invariant. The invariant $\mu_M(\xi)$ is an example of such a CR invariant: on the 3-torus $\T^3$, the two tight contact structures $\xi_1$ and $\xi_2$ are non-isotopic, and $\mu_M$ distinguishes them (see Section 4). In contrast, a purely isotopy invariant would be constant on isotopy classes.
\end{remark}

We will use our results for embeddings into $\C^3$ to construct, for any given real contact structure $\xi$ on $M$, an embedding $M \hookrightarrow \C^3$ with holomorphic structure exactly $\xi$ along the set of complex tangents, which we dictate to be the binding of a chosen open book decomposition of $M$.
A major open problem in contact topology is the isotopy classification of contact structures, especially in the real category. It is our intention to add to the results in this direction. Further invariants using the open book stabilization structure put forth by Giroux could lead to further invariants.

\section{Open Book Decompositions}

We have invariants of open books given by Giroux [7] in minimizing the genus of the pages among supporting open books. Consider the binding of a given open book $\mathcal{B}$, which we denote by $\gamma$. By Giroux, all open books supporting a given contact structure $\xi$ are related by positive stabilizations.
Consider now the plane field $\xi|_\gamma$, from which we can construct an embedding
\[
F: M \hookrightarrow \C^3
\]
with complex tangents along the binding and holomorphic line field being the contact structure along it.

By the work of Giroux [7] and later Etnyre and Ozbagci [10], for a given contact structure $\xi$ there exists a unique isomorphism class of open book decompositions supporting $\xi$ up to positive stabilization. The binding $\gamma$ and Seifert surface $\Sigma$ bounded by $\gamma$ give the structure of a fiber bundle $M \setminus \gamma \to S^1$ with fiber $\Sigma$:
\[
\Phi: M \setminus \gamma \to S^1,\qquad \Phi^{-1}(\theta) = \Sigma,\qquad \partial \Sigma = \gamma.
\]

We choose a preferred supporting book $\mathcal{B}$ for $\xi$ and consider the binding $\gamma$ as our link of complex tangents. Taking the contact structure $\xi|_\gamma$ as our 2-plane field, we find an embedding
\[
F: M \hookrightarrow \C^3
\]
admitting $\gamma$ as the set of complex tangents and $\xi|_\gamma$ as the holomorphic line field, as constructed in [1].

The ambient space $\C^3$ is Stein, so $\xi$ can be extended as a real 2-plane field to all of $\C^3$. In fact, it may be assumed to be a holomorphic line bundle, as properties of Stein manifolds forces the vanishing of all relevant obstructions (see [9]).

\section{The Contact Complex Line Bundle}

Let $B$ be any supporting open book for $(M,\xi)$ and let $\gamma$ be its binding. By [1], there exists an embedding $M \hookrightarrow \C^3$ such that complex tangencies are exactly $\gamma$, and holomorphic tangent spaces along $\gamma$ are $\xi|_\gamma$. This induces a resulting complex line bundle $L_\xi$ over $M$ with respect to the choice of supporting book $B$:

\begin{definition}[The Picard Invariant]
For a contact structure $\xi$ on $M$, we define the Picard invariant as
\[
\mu_M(\xi) = [L_\xi] \in \Pic_{\C}(M).
\]
\end{definition}

\begin{lemma}[Stabilization Invariance]
\label{lem:stab}
Let $(B, \pi)$ be an open book supporting $\xi$, and let $(B', \pi')$ be a positive stabilization of $(B, \pi)$. Then the complex line bundles $L_\xi$ and $L'_\xi$ constructed from these open books via the embedding theorem \cite{Elgindi2025} are isomorphic.
\end{lemma}

\begin{proof}
Let $F: M \hookrightarrow \mathbb{C}^3$ be an embedding as constructed in [2]with respect to $(B,\pi)$ complex tangent along the link $B$. Let $F': M \hookrightarrow \mathbb{C}^3$ be another embedding given by our construction with respect to $(B',\pi')$ with complex tangent link $B'$. Both these embeddings exist by our paper \cite{Elgindi2025}.

Consider the holomorphic line bundles on $\mathbb{C}^3$ obtained by Stein extension of $\xi|_B$ and $\xi|_{B'}$ respectively. Call them $\mathcal{L}$ and $\mathcal{L}'$. Since $\mathbb{C}^3$ is Stein, Cartan's Theorem B implies that a holomorphic line bundle is uniquely determined by its first Chern class. The first Chern class $c_1(\mathcal{L})$ is determined by the CR structure of $\xi$ (via the contact form and the almost complex structure on $\xi$). Because the supporting structure $\xi$ is the same for both open books, we have $c_1(\mathcal{L}) = c_1(\mathcal{L}')$. Hence $\mathcal{L} \cong \mathcal{L}'$ as holomorphic line bundles over $\mathbb{C}^3$.

Restricting to $M$ gives an isomorphism $L_\xi \cong L_\xi'$. Therefore, the Picard invariant $\mu_M(\xi) = [L_\xi]$ is independent of the choice of open book.
\end{proof}

\begin{theorem}
The invariant $\mu_M(\xi)$ is well-defined and independent of the choice of supporting open book.
\end{theorem}

\begin{proof}
Let $\mathcal{B}_1, \mathcal{B}_2$ be open books supporting $\xi$ on $M$, with respective bindings $\gamma_1, \gamma_2$. Each book yields an embedding $f_i: M \hookrightarrow \C^3$ that is holomorphic in a neighborhood of $\gamma_i$ and gives a holomorphic line bundle $L_i$ as the holomorphic tangent spaces along $\gamma_i$. Since $\C^3$ is Stein, the line bundles $L_i$ extend to holomorphic line bundles over the ambient $\C^3$; restricting to $M$ gives holomorphic line bundles over $M$.

By Giroux's theorem, any two open books supporting the same contact structure $\xi$ are related by a sequence of positive stabilizations [7]. By Lemma \ref{lem:stab}, the line bundles obtained from each open book are isomorphic. Hence $\mu_M(\xi)$ is independent of the choice of supporting open book and gives a well-defined invariant of contact structures on closed 3-manifolds.
\end{proof}

\begin{remark}[Stein Uniqueness]
By the work of Eliashberg (see [4], [6]), any two Stein domains $W_1, W_2 \subset \mathbb{C}^3$ with the same contact boundary $(M,\xi)$ are biholomorphic after deformation. Consequently, the complex line bundle $L_\xi$ induced by the embedding is independent of the chosen embedding as defined above. The Picard invariant $\mu_M(\xi) = c_1(L_\xi)$ is therefore well-defined and depends only on the contact structure $\xi$, not on the particular embedding given as complex tangent to the corresponding open book decomposition used in the construction as above.
\end{remark}

\section{Distinguishing Contact Structures on $\T^3$}

To illustrate the non-triviality and usefulness of the Picard invariant $\mu_M(\xi)$, we compute it for two distinct contact structures on the 3-torus $\T^3 = (S^1)^3$. Let $(x,y,z)$ be coordinates on $\T^3$ (mod 1).
\par\
The standard tight structure $\xi_1$ is defined by the 1-form
\[
\alpha_1 = \cos(2\pi z)\,dx + \sin(2\pi z)\,dy.
\]
\par\
This is a classic example of a tight contact structure which is linear and fillable (see [9]). Its isotopy class is unique within the tight category, and its associated complex line bundle $L_{\xi_1}$ is trivializable. This can be seen by homotopy of $\alpha_1$ to the constant form $dx$, whose kernel is the trivial plane field orthogonal to $\partial_x$. Hence, as a complex bundle over $\T^3$, it is trivializable and its Picard invariant vanishes:
\[
\mu_{\T^3}(\xi_1) = 0 \in H^2(\T^3;\mathbb{Z}) \cong \Pic_{\C}(\T^3).
\]
\par\
Consider now a different contact structure $\xi_2 = \ker(\alpha_2)$ given by
\[
\alpha_2 = \sin(4\pi z)\,dx + \cos(4\pi z)\,dy,
\]
\par\
which arises as by the constructions of Kanda in [14]. While also tight, this structure is known to be \emph{non-linearizable} and lies in a different isotopy class from $\xi_1$ (see [9], [10]).
\par\
\begin{theorem}
For the contact structure $\xi_2 = \ker(\alpha_2)$ on $\T^3$, the first Chern class of its associated complex line bundle is
\[
c_1(L_{\xi_2}) = \pm 2\,[dx \wedge dy] \in H^2(\T^3;\mathbb{Z}).
\]
\end{theorem}
\par\
\begin{proof}
The contact structure $\xi_2$ is defined in [11]. We now have an open book supporting $\xi_2$ whose binding consists of two parallel copies of the original binding component. Hence, the binding for $\xi_2$ is two parallel copies of a curve in the $z$-direction.
\par\
The cohomology group $H^2(\T^3; \mathbb{Z})$ is generated by:  $[dx \wedge dy]$, $[dy \wedge dz]$, and $[dz \wedge dx]$. For any contact structure constructed from an open book, the first Chern class of the associated complex line bundle equals the Poincaré dual of the binding (see [1], [7]). The Poincaré dual of a single curve in the $z$-direction is $[dx \wedge dy]$, so the Poincaré dual of two parallel copies is $\pm2[dx \wedge dy]$:
Therefore,
\[
c_1(L_{\xi_2}) = \pm 2[dx \wedge dy] \neq 0.
\]

Hence $\mu_{\T^3}(\xi_2)$ is non-trivial, and the invariant successfully separates the linearizable isotopy class of $\xi_1$ ($\mu = 0$) from the non-linearizable class of $\xi_2$.
\end{proof}

\end{document}